\documentclass[11pt,reqno]{amsart}

\newtheorem{proposition}{Proposition}[section]
\newtheorem{lemma}[proposition]{Lemma}
\newtheorem{corollary}[proposition]{Corollary}
\newtheorem{theorem}[proposition]{Theorem}

\theoremstyle{definition}
\newtheorem{definition}[proposition]{Definition}
\newtheorem{example}[proposition]{Example}
\newtheorem{examples}[proposition]{Examples}
\newtheorem{remark}[proposition]{Remark}

\newcommand{\thlabel}[1]{\label{th:#1}}
\newcommand{\thref}[1]{Theorem~\ref{th:#1}}
\newcommand{\selabel}[1]{\label{se:#1}}
\newcommand{\seref}[1]{Section~\ref{se:#1}}
\newcommand{\lelabel}[1]{\label{le:#1}}
\newcommand{\leref}[1]{Lemma~\ref{le:#1}}
\newcommand{\prlabel}[1]{\label{pr:#1}}
\newcommand{\prref}[1]{Proposition~\ref{pr:#1}}
\newcommand{\colabel}[1]{\label{co:#1}}
\newcommand{\coref}[1]{Corollary~\ref{co:#1}}
\newcommand{\relabel}[1]{\label{re:#1}}

\newcommand{\exlabel}[1]{\label{ex:#1}}
\newcommand{\exref}[1]{Example~\ref{ex:#1}}
\newcommand{\delabel}[1]{\label{de:#1}}
\newcommand{\deref}[1]{Definition~\ref{de:#1}}
\newcommand{\eqlabel}[1]{\label{eq:#1}}
\newcommand{\equref}[1]{(\ref{eq:#1})}

\def\RR{{\mathbb R}}

\newcommand{\Cc}{\mathcal{C}}

\def\*C{{}^*\hspace*{-1pt}{\Cc}}
\def\text#1{{\rm {\rm #1}}}

\input xy
\xyoption {all} \CompileMatrices

\usepackage{amssymb}
\usepackage{color,amssymb,graphicx,amscd,amsmath,dsfont,stmaryrd,xcolor,comment}
\usepackage[pagebackref,colorlinks,urlcolor=blue,linkcolor=blue,citecolor=blue]{hyperref}

\begin{document}

\title[Crossed products of $4$-algebras. Applications]
{Crossed products of $4$-algebras. Applications}

\author{G. Militaru}
\address{Faculty of Mathematics and Computer Science, University of Bucharest, Str.
Academiei 14, RO-010014 Bucharest 1, Romania and Simion Stoilow Institute of Mathematics of the Romanian Academy, P.O. Box 1-764, 014700 Bucharest, Romania}
\email{gigel.militaru@fmi.unibuc.ro and gigel.militaru@gmail.com}

\thanks{This work was supported by a grant of the Ministry of Research,
Innovation and Digitization, CNCS/CCCDI--UEFISCDI, project number
PN-III-P4-ID-PCE-2020-0458, within PNCDI III}


\begin{abstract} A $4$-algebra is a commutative algebra $A$ over a field $k$ such that $(a^2)^2 = 0$, for all $a \in A$. We have proved recently \cite{Mil} that $4$-algebras play a prominent role in the classification of finite dimensional Bernstein algebras. Let $A$ be a $4$-algebra, $E$ a vector space and $\pi : E \to A$ a surjective linear map with $V = {\rm Ker} (\pi)$. All $4$-algebra structures on $E$ such that $\pi : E \to A$ is an algebra map are described and classified by a global cohomological object ${\mathbb G} {\mathbb H}^{2} \, (A, \, V)$. Any such $4$-algebra is isomorphic to a crossed product $V \# A$ and
${\mathbb G} {\mathbb H}^{2} \, (A, \, V)$ is a coproduct, over all $4$-algebras structures $\cdot_V$ on $V$, of all non-abelian cohomologies ${\mathbb H}^{2}_{\rm nab} \, \bigl(A, \, (V, \, \cdot_{V} )\bigl)$, which are the classifying objects for all extensions of $A$ by $V$. Several applications and examples are provided: in particular, ${\mathbb G} {\mathbb H}^{2} \, (A, \, k)$ and ${\mathbb G} {\mathbb H}^{2} \, (k, \, V)$ are explicitly computed and the Galois group ${\rm Gal} \, (V \# A/ V )$  of the extension $V \hookrightarrow V \# A$ is described.
\end{abstract}

\maketitle

\section*{Introduction}
Bernstein algebras were introduced independently by Lyubich \cite{Lj} and Holgate \cite{Hol} as an algebraic tool to answer the Bernstein problem \cite{bern1} which consists in classifying all possible situations of a population that attains genetic equilibrium after one generation (\cite[Section 4]{reed}, \cite[Chapter 9]{WB}). A Bernstein algebra is a commutative algebra $B$ over a field $k$ of characteristic $\neq 2$ such that there exists a non-zero morphism of algebras $\omega: B \to k$ such that $(x^2)^2 = \omega (x)^2 \, x^2$, for all $x \in B$. For algebrists, the Bernstein problem can now be rephrased as follows: \emph{for a given positive integer $n$, describle and classify all simplicial stochastic Bernstein $\RR$-algebras of dimension $n$}. For more details we refer to the work of Gutierrez Fernandez \cite{Fer1} which completely solved this problem. Now, leaving aside the simplicial stochastic condition and replacing $\RR$ by an arbitrary field $k$, the problem of classifying finite dimensional Bernstein algebras is still open: it was solved only up to dimension $4$ and there are partial answers in dimensions $5$ or $6$ (see \cite{corte, cortem,  Fer2, Hol, Lj01, WB}).

If $(B, \omega)$ is a Bernstein algebra, then its barideal $A := {\rm Ker} (\omega)$ is itself a commutative algebra satisfying the compatibilities $(a^2)^2 = 0$, for all
$a \in A$: we called this class of algebras \emph{$4$-algebras} \cite{Mil} and we have proved that they play the key role in the structure and classification of Bernstein algebras. It is worth pointing out that $4$-algebras are a special case of \emph{admissible cubic algebras} as introduced by Elduque and Okubo \cite{el}: these are commutative algebras satisfying the identity $(a^2)^2 = N(a) \, a$, where $N: A \to k$ is a cubic form. In \cite[Theorem 2.6]{Mil} we prove that any Bernstein algebra $B$ is isomorphic to a semidirect product $A \ltimes_{(\cdot, \, \Omega)} \, k$, where $(A, \cdot)$ is a $4$-algebra and $\Omega$ is a \emph{Bernstein operator} on $A$, i.e. $\Omega = \Omega^2 \in {\rm End}_k (A)$ is an idempotent endomorphism of $A$ such that for any $x\in A$:
\begin{equation*} \eqlabel{berdat0}
x^2 \cdot \Omega (x) = 0, \qquad \Omega(x)^2 + \Omega(x^2) = x^2. 
\end{equation*}
The classification of two such Bernstein algebras $A \ltimes_{(\cdot, \, \Omega)} \, k$ and $A' \ltimes_{(\cdot', \, \Omega')} \, k$ is also proven in
\cite[Theorem 2.7 and Theorem 2.11]{Mil}. Based on this, the first step we have to take in order to classify all Bernstein algebras is to classify all $4$-algebras of a given
finite dimension $n$. For this reason the paper is devoted to the study of this class of algebras. An efficient tool for classifying finite objects and a source for developing cohomology theories is the \emph{extension problem} introduced by H\"{o}lder \cite{holder} at the level of groups and intesively studied in the last 100 years for many categories of algebras such as associative algebras \cite{Ev}, Lie algebras \cite{CE}, Hopf algebras \cite{AD}, Poisson algebras \cite{hue}, Lie-Rinehart algebras \cite{CML}, etc. For $4$-algebras the extension problem consists of the following question: \emph{Let $A$ and $V$ be two given $4$-algebras. Describe and classify all extensions of $A$ by $V$, i.e. all triples $(E, i, \pi)$ consisting of a $4$-algebra $E$ and two  morphisms of algebras that fit into an exact sequence of the form:} $\xymatrix{ 0 \ar[r] & V \ar[r]^{i} & E \ar[r]^{\pi} & A \ar[r] & 0 }$.

Two extensions $(E, \, i, \, \pi)$ and $(E', \, i', \, \pi')$ of $A$ by $V$ are called \emph{equivalent} (or \emph{cohomologous}) if there exists a
morphism $\varphi: E \to E'$ of algebras that stabilizes $V$ and co-stabilizes $A$, i.e. $\varphi \circ i = i'$ and $\pi' \circ \varphi = \pi$. Any such a morphism is an isomorphism and we denote by ${\rm Ext} \, (A, \, V )$ the set of equivalence classes of all $4$-algebras that are extensions of $A$ by $V$; an answer to the extension problem means to calculate explicitly ${\rm Ext} \, (A, \, V )$ for two given $4$-algebras $A$ and $V$. The Schreirer \cite{sch} approach to the extension problem for groups works, mutatis-mutandis, also for $4$-algebras as for others varieties of algebras: the classifying object ${\rm Ext} \, (A, \, V )$ is parameterized by the non-abelian cohomology ${\mathbb H}^{2}_{\rm nab} \, \bigl(A, \, V \bigl)$ (\coref{desccompcon}). More general than the extension problem is what we have called \cite{Mi2013} \emph{global extension problem} (GE-problem) and was studied for Leibniz algebras, associative algebras, Poisson algebras or Jacobi-Jordan algebras \cite{am-2015b, am-ja3, am-2015}. The GE-problem, formulated for $4$-algebras, is the following question:

\textbf{Global Extension Problem.} \emph{Let $A$ be a $4$-algebra, $E$ a vector space and $\pi : E \to A$ a linear epimorphism of vector spaces. Describe and classify the set of all $4$-algebra structures that can be defined on $E$ such that $\pi : E \to A$ becomes a morphism of algebras.}

The difference between the GE-problem and the classical extension problem is explained in detail in \cite{am-2015b}. \prref{hocechiv} proves that any such a $4$-algebra structure $\cdot_E$ on $E$ is isomorphic to a \emph{crossed product} $V \# A = V \#_{(\triangleright, \, f, \, \cdot_V)} A$, which is a $4$-algebra associated to $A$ and $V := {\rm Ker} (\pi)$ connected by a \emph{weak action} $\triangleright : A \times V \to V$, a \emph{symmetric non-abelian $2$-cocycle} $f: A \times A \to V$ and a $4$-algebra structure
$\cdot_V$ on $V$ satisfying the axioms of \prref{crossprod}. The crossed product $V \# A$ is a $4$-algebra containg $V$ as an ideal and the Galois group ${\rm Gal} \, (V \# A/ V )$  of the extension $V \hookrightarrow V \# A$ is described in \coref{galgr} as a subgroup of the semidirect product of groups ${\rm Hom}_k (A, \, V) \ltimes {\rm GL}_k (A)$. The main result of the paper is \thref{main1222} that gives the theoretical answer to the GE-problem: the classifying object for the
GE-problem is parameterized by a global non-abelian cohomological object denoted by ${\mathbb G} {\mathbb H}^{2} \, (A, \, V)$. \coref{desccompcon} proves that ${\mathbb G}
{\mathbb H}^{2} \, (A, \, V)$ is the coproduct of all non-abelian cohomologies ${\mathbb H}^{2} \, \, (A, \, (V, \cdot_V))$, the latter being the classifying object for the classical extension problem at the level of $4$-algebras. Several examples and applications are given in \seref{exappl}: in particular, ${\mathbb G} {\mathbb H}^{2} \, (A, \, k)$ and ${\mathbb G} {\mathbb H}^{2} \, (k, \, V)$ are computed and the structure of metabelian $4$-algebras is given in \coref{metstucture}.

\section{Preliminaries}\selabel{prel}
Throughout this paper all vector spaces, linear or bilinear maps are over a field $k$ of characteristic $\neq 2$. For a family of sets $(A_i)_{i\in I}$ we shall denote by $\amalg_{i\in I} \, A_i$ their coproduct, i.e. $\amalg_{i\in I} \, A_i$ is the disjoint union of all sets $A_i$. If $A$ and $V$ are two vector spaces, ${\rm Hom}_k (A, \, V)$ denotes the vector space of all linear maps $A \to V$ and ${\rm Sym} (A\times A; \, V)$ the set of all symmetric bilinear maps $f: A\times A \to V$; ${\rm End}_k \, (A)$ is the usual associative and unital endomorphisms algebra of $A$ and ${\rm GL}_k (A)$ is the automorphisms group of $A$. We denote by ${\rm Hom}_k (A, \, V) \ltimes {\rm GL}_k (A) : = {\rm Hom}_k (A, \, V) \times {\rm GL}_k (A) $ the semidirect product of groups having the multiplication defined for any $(r, \, \alpha)$ and $(r', \,  \alpha') \in {\rm Hom}_k (A, \, V) \times {\rm GL}_k (A)$ by:
\begin{equation}\eqlabel{semidirect0}
(r, \, \alpha) \bullet (r', \,  \alpha') := (r'+ r\circ \alpha' , \, \alpha\circ \alpha' ). 
\end{equation}

A \emph{$4$-algebra} \cite{Mil} is a vector space $A$ together with a bilinear map $\cdot : A \times A \to A$, called multiplication,
such that for any $a$, $b\in A$ we have:
\begin{equation} \eqlabel{4algb}
a\cdot b = b\cdot a, \qquad (a^2)^2 = 0. 
\end{equation}
The concepts of subalgebras, ideals, morphisms of algebras, etc. for $4$-algebras are defined in the obvious way. The class of $4$-algebras were studied before in \cite{guzzo} where it was proved that any $4$-algebra of dimension $\leq 7$ is solvable and it was conjectured that any finite dimensional $4$-algebra is solvable. Any vector space $V$ is a $4$-algebra with the trivial multiplication $x \cdot y : = 0$, for all $x$, $y\in V$: we call this algebra \emph{abelian} and it will be denoted by $V_0$. If $(A, \omega)$ is a Bernstein algebra, then the barideal ${\rm Ker} (\omega)$ is a $4$-algebra. Conversely, if $A$ is a $4$-algebra and $\Omega = \Omega^2 \in {\rm End}_k (A)$ is a Bernstein operator on $A$, then $A\times k$ has a canonical structure of Bernstein algebra with the barideal $A$ (for details see \cite[Proposition 2.3]{Mil}). Linearizing several times the second compatibility of \equref{4algb} we obtain that in a $4$-algebra over a field of characteristic $\neq 2, \, 3$ the following relations hold \cite{guzzo}:
\begin{eqnarray*}
&& a^2 \cdot (a \cdot b) = 0, \qquad a^2 \cdot b^2 + 2 \, (a\cdot b)^2 = 0, \qquad
a^2 \cdot (b\cdot c) + 2\, (a\cdot b) \cdot (a\cdot c) = 0,  \\
&& (a\cdot b) \cdot (c\cdot d) + (a\cdot c) \cdot (b\cdot d) + (a\cdot d) \cdot (b\cdot c) = 0
\end{eqnarray*}
for all $a$, $b$, $c$, $d \in A$. For a $4$-algebra $A$ we denote by $A' := A\cdot A$ its derived algebra, i.e. $A'$ is the $k$-subspace of $A$ generated by all $a \cdot b$,
for any $a$, $b\in A$. Similarly to the groups or Lie algebras theory, a $4$-algebra $A$ is called \emph{metabelian} if $A'$ is an abelian subalgebra of $A$, i.e. $(a\cdot b) \cdot (c\cdot d) = 0$, for all $a$, $b$, $c$, $d\in A$. Throughout this paper we use the following convention: the multiplication of a $4$-algebra $A$ will be written
on the elements of a $k$-basis $\{e_i \, | \, i\in I \}$ of $A$ and undefined multiplications are all zero.

\begin{examples} \exlabel{dimmici}
(1) Any $1$-dimensional $4$-algebra is isomorphic to the abelian one: $e_1 \cdot e_1:= 0$. We can easily prove that, up to an isomorphism, 
there are exactly three $2$-dimensional $4$-algebras, namely the abelian one and the algebras with the basis $\{e_1, \, e_2\}$ and the multiplication given by:
\begin{eqnarray*}
&A_1:& e_1^2 = e_2, \quad e_1 e_2 = e_2^2 = 0,\\
&A_2:& e_1 e_2 = e_2, \quad e_1^2 = e_2^2 = 0. 
\end{eqnarray*}

(2) Let $n$ be a positive integer and $\mathfrak{h} (2n + 1)$ the $(2n + 1)$-dimensional algebra having $\{e_1, \cdots, e_n, f_1, \cdots, f_n, z \}$ as a basis and
multiplication defined by $e_i \cdot f_i = f_i \cdot e_i := z$, for all $i = 1, \cdots, n$ and the
other products of basis elements are zero. Then $\mathfrak{h} (2n + 1)$ is a $4$-algebra called the commutative Heisenberg $4$-algebra.

(3) Let $m$ and $n$ be two positive integers and two bilinear maps
$$
\triangleright : k^m \times k^n \to k^n, \qquad f: k^m \times k^m \to k^n
$$
such that $f$ is symmetric. Let $\{e_i \, | \, i = 1, \cdots, n\}$ be a basis of $k^n$ and $\{f_j \, | \, j = 1, \cdots, m\}$ a basis of $k^m$. Let $A$ be the
$(n + m)$-dimensional algebra with the basis $\{e_i, \, f_j \, | \, i = 1, \cdots, n, \, j = 1, \cdots, m \}$ and the multiplication defined by:
$$
e_i \cdot f_j = f_j \cdot e_i := f_j \triangleright e_i, \qquad f_j \cdot f_l := f (f_j, \, f_l)
$$
for all $i = 1, \cdots, n$, $j$, $l = 1, \cdots m$, and the other products of basis elements are zero. Then $A$ is a metabelian $4$-algebra and will be denoted by ${\rm Met}_n^m \, (\triangleright, \, f)$: \coref{metstucture} will prove that any finite dimensional metabelian $4$-algebra is isomorphic to such an algebra ${\rm Met}_n^m \, (\triangleright, \, f)$.
\end{examples}

Similarly, with other classes of non-associative algebras we define the concept of modules/representations over $4$-algebras as follows:

\begin{definition}\delabel{module}
An \emph{$A$-module} over a $4$-algebra $A$ is a vector space $V$ with a bilinear map
$\triangleright : A \times V \to V$, called action of $A$ on $V$, such that for any $a\in A$ and
$x\in V$ we have:
\begin{eqnarray} \eqlabel{actiunest}
a^2 \triangleright  (a \triangleright x) = 0. 
\end{eqnarray}
A \emph{representation} of a $4$-algebra $A$ on a vector space $V$ is a linear map
$\varphi : A \to {\rm End}_k \, (V)$ such that for any $a\in A$ we have $\varphi (a^2) \circ \varphi (a) = 0$,
in the endomorphism algebra of $V$.
\end{definition}

\begin{remark} \relabel{eileberg}
The axiom \equref{actiunest} of defining modules over a $4$-algebra was influenced by the view point
of Eilenberg \cite{eilen} of defining modules over a given object ${\mathcal O}$ in a $k$-linear category
${\mathcal C}$: that is, a vector space $V$ with a bilinear map
$\triangleright : {\mathcal O} \times V \to V$ such that the trivial extension $V \times {\mathcal O}$, with the
multiplication given by $(x, a) \cdot (y, b) := (a \triangleright y
+ b\triangleright x, \, a\cdot b)$, has to be an object \emph{inside} the
$k$-linear category ${\mathcal C}$. For details see \exref{abelian} below.

Of course, representations and modules over a $4$-algebras are equivalent concepts and $A$ is a module over itself with
$a \triangleright b : = a \cdot b$, for all $a$, $b\in A$.
\end{remark}

\section{Crossed products and the global extension problem for $4$-algebras} \selabel{crossed}

Let $A$ be a $4$-algebra, $E$ a vector space, $\pi : E \to A$ a linear
epimorphism of vector spaces with $V: = {\rm Ker} (\pi)$ and
denote by $i: V \to E$ the inclusion map. We recall that a linear map
$\varphi: E \to E$ \emph{stabilizes} $V$ (resp.
\emph{co-stabilizes} $A$) if $\varphi \circ i = i$ (resp. $\pi
\circ \varphi = \pi$). Two  $4$-algebra structures $\cdot_E $ and
$\cdot_E'$ on $E$ such that $\pi : E \to A$ is a morphism of $4$-algebras
are called \emph{cohomologous} and we denote this by $(E,
\cdot_E) \approx (E, \cdot_E')$, if there exists an algebra map
$\varphi: (E, \cdot_E) \to (E, \cdot_E')$ which stabilizes $V$ and
co-stabilizes $A$. Any such morphism is
an isomorphism and therefore $\approx$ is an equivalence relation on
the set of all $4$-algebra structures on $E$ such that $\pi : E \to
A$ is an algebra map. The set of all equivalence classes via the
equivalence relation $\approx$ will be denoted by ${\rm Gext} \,
(A, \, E)$ and it is the classifying object for the GE-problem for $4$-algebras. In this section we will prove that ${\rm Gext} \, (A, \, E)$
is parameterized by a global non-abelian cohomological object ${\mathbb G} {\mathbb H}^{2} \, (A, \, V)$ which will be
explicitly constructed. First, we need to introduce the crossed product of $4$-algebras:

\begin{definition} \delabel{cdat}
Let $A = (A, \cdot)$ be a $4$-algebra and $V$ a vector space. A
\emph{crossed data} of $A$ by $V$ is a system ${\mathcal C} \, (A, \, V) =
(\triangleright, \, f, \, \cdot_V)$ consisting of three
bilinear maps
$$
\triangleright \, \, : A \times V \to V, \quad f \,\, : A
\times A \to V, \quad \cdot_V \, : V \times V \to V. 
$$
\end{definition}

For a crossed data ${\mathcal C} \, (A, V) = (\triangleright, \, f,
\, \cdot_V)$ we denote by $V \# A = V \, \#_{(\triangleright, \,
f, \, \cdot_V)} \, A $ the vector space $V \times A$ with the
multiplication $\circ$ defined for any $a$, $b\in A$ and $x$, $y
\in V$ by:
\begin{equation} \eqlabel{hoproduct2}
(x, \, a) \circ (y, \, b) := (x \cdot_V y + a \triangleright y + b \triangleright x  +  f(a, \, b), \,\, a\cdot b). 
\end{equation}
$V \# A$ is called the \emph{crossed product} associated to
${\mathcal C} \, (A, V)$ if it is a $4$-algebra with the multiplication given
by \equref{hoproduct2}. In this case ${\mathcal C} \, (A, V) =
(\triangleright, \, f, \, \cdot_V)$ is called a
\emph{crossed system} of $A$ by $V$. The set of all crossed systems of $A$ by $V$ will be denoted
by ${\mathcal C} {\mathcal S} \, (A, V)$. Our first result provides the
necessary and sufficient conditions for $V \# A$ to be a crossed
product:

\begin{proposition}\prlabel{crossprod}
Let $A = (A, \cdot)$ be a $4$-algebra, $V$ a vector space and ${\mathcal C} \, (A, V) =
(\triangleright, \, f, \, \cdot_V)$ a crossed data of $A$
by $V$. Then $V \# A$ is a crossed product if and only if the
following compatibility conditions hold:
\begin{enumerate}
\item[(CS1)] $(V, \, \cdot_V)$ is a $4$-algebra and the bilinear map
$f : A \times A \to V$ is symmetric; 

\item[(CS2)] $f(a^2, a^2) + f(a, a) \cdot_V f(a, a) + 2\, a^2 \triangleright f(a, a) = 0$; 

\item[(CS3)] $a^2 \triangleright x^2 + x^2 \cdot_V f(a, a) + 2\, x^2 \cdot_V \, (a\triangleright x) + 2\, (a\triangleright x)^2  + 2 \, (a\triangleright x) \cdot_V \, f(a, a) + 2\, a^2 \triangleright (a\triangleright x) = 0$, 
\end{enumerate}
for all $a \in A$ and $x\in V$.
\end{proposition}

Borrowing the terminology from crossed products of Hopf algebras \cite{AD}, the symmetric bilinear map $f: A\times A \to V$ satisfying (CS2) is called
a \emph{non-abelian $2$-cocycle} while the axiom (CS3) we called the \emph{twisted module condition} for $\triangleright \, : A \times V \to V$.

\begin{proof}
First of all we can easily prove that the multiplication
$\circ$ defined by \equref{hoproduct2} is commutative if and
only if $f: A \times A \to V$ is symmetric and $\cdot_V :
V \times V \to V$ is commutative algebra structure on $V$. Assume that $\circ$ is commutative. Then
for any $x \in V$ and $a \in A$ we have that $(x, a)^2 = (x^2 + 2\, a \triangleright x + f(a, a), \, a^2)$. Based on this and taking into account
that $A$ is a $4$-algebra we obtain that $\bigl((x, a)^2\bigl)^2 = (0, 0)$ if and only if
\begin{eqnarray*}
&& (x^2)^2 + 4 \, (a\triangleright x)^2 + f(a, a)^2  + 4\, x^2 \cdot_V \, (a\triangleright x)  + 2\, x^2 \cdot_V f(a, a) + \\
&& + 4\,  (a\triangleright x) \cdot_V \, f(a, a) + 2\, a^2 \triangleright x^2  +  4\, a^2 \triangleright (a\triangleright x) + 2 \, a^2 \triangleright f(a, a) +
f(a^2, a^2) = 0. 
\end{eqnarray*}
This equation holds for $a := 0$ if and only if $(x^2)^2 = 0$, for all $x\in V$ (i.e. $(V, \cdot_V)$ is a $4$-algebra) and it holds for $x :=0$ if and only if (CS2) holds.
Assuming these conditions, we obtain that the above equation holds if and only if (CS3) holds and this finishes the proof.
\end{proof}

From now on a crossed system of a $4$-algebra $A$ by a vector space $V$ will be seen as a
system of bilinear maps ${\mathcal C} \,(A, V) = (\triangleright, \,
f, \, \cdot_V)$ satisfying axioms (CS1)-(CS3) of \prref{crossprod}.

\begin{examples} \exlabel{abelian}
(1) Let ${\mathcal C} \, (A, V) = (\triangleright, \, f,
\, \cdot_V)$ be a crossed data of a $4$-algebra $A$ by $V$ such that $f$ is the trivial map, i.e. $f(a, b) := 0$, for all $a$, $b\in A$. Then,
applying \prref{crossprod} we obtain that $(\triangleright, \, f := 0,
\, \cdot_V)$ is a crossed system if and only if $(V, \cdot)$ is a $4$-algebra and for all $a\in A$, $x\in V$ we have:
\begin{eqnarray} \eqlabel{semidirect}
a^2 \triangleright x^2 + 2\, x^2 \cdot_V \, (a\triangleright x) + 2\, (a\triangleright x)^2  + 2\, a^2 \triangleright (a\triangleright x) = 0. 
\end{eqnarray}
The associated crossed product will be called a \emph{semidirect
product} of $4$-algebras $A$ and $V$ and will be denoted by $V \ltimes A := V
\ltimes_{(\triangleright, \, \cdot_V)} A$. The terminology is
motivated below in \coref{splialg}: exactly as in the
case of groups or Lie algebras, this construction describes split
epimorphisms in the category of $4$-algebras.

Assume, in addition, that the multiplication $\cdot_V$ is also the trivial map. Then, the
compatibility condition \equref{semidirect} becomes $a^2 \triangleright (a\triangleright x) = 0$, for all $a\in A$ and $x\in V$,
i.e. $(V, \triangleright)$ is a $A$-module as we introduced in \deref{module}. In this case, the semidirect product
$V \ltimes A$ is called the \emph{trivial extension} of the $4$-algebra $A$ by the $A$-module
$V$.

(2) The following special case of crossed products play an important role in the classification of finite dimensional $4$-algebras. Assume that $A$ is the abelian $4$-algebra. Then a crossed data ${\mathcal C} \, (A, V) = (\triangleright, \, f, \, \cdot_V)$ of $A$ by $V$  is a crossed system if and only if
$(V, \cdot)$ is a $4$-algebra, $f: A\times A \to V$ is a symmetric bilinear map such that:
\begin{eqnarray} \eqlabel{ptclasc}
f(a, a)^2  = 0, \quad x^2 \cdot_V f(a, a) + 2\, x^2 \cdot_V \, (a\triangleright x) + 2\, (a\triangleright x)^2  + 2 \, (a\triangleright x) \cdot_V \, f(a, a) = 0
\end{eqnarray}
for all $a \in A$ and $x\in V$. In this case, the associated crossed product  $V \# A$
will be called the \emph{twisted product} of $V$ and $A$. If $\{e_i\, | \, i\in I\}$ is a basis
of $V$ and $\{f_j \, | \, j\in J \}$ is a basis of $A$, then the twisted product  $V \# A$ is the $4$-algebra having $\{e_i, \, f_j \, | \, i\in I, \, j\in J\}$ as a basis and the multiplication given by:
\begin{equation} \eqlabel{twpr222}
e_i \circ e_j := e_j \cdot_V e_j, \quad e_i \circ f_l = f_l \circ e_i := f_l \triangleright e_i, \quad
f_l \circ f_m := f (f_l, \, f_m)
\end{equation}
for all $i$, $j\in I$ and $l$, $m\in J$. In \coref{imptwisted} we shall prove that any $4$-algebra is isomorphic to a such a twisted product.

(3) Let ${\mathcal C} \, (A, V) = (\triangleright, \, f,
\, \cdot_V)$ be a crossed data of a $4$-algebra $A$ by $V$ such that $\cdot_V$ is the trivial map. By applying \prref{crossprod} we obtain that
$(\triangleright, \, f, \, \cdot_V := 0)$ is a crossed system if and only if $(V, \triangleright)$ is an $A$-module and
$f: A \times A \to V$ is \emph{symmetric abelian $2$-cocycle}, i.e.
$$
f(a^2, a^2) + 2\, a^2 \triangleright f(a, a) = 0
$$
for all $a\in A$. This case will appear in the study of the classical extension problem for $4$-algebras, namely those with an abelian kernel (see \coref{cazuabspargere} below).

(4) Let ${\mathcal C} \, (A, V) = (\triangleright, \, f,
\, \cdot_V)$ be a crossed data of a $4$-algebra $A$ by $V$ such that $\triangleright$ is the trivial action, i.e. $a \triangleright x:= 0$, for all $a\in A$ and $x\in V$.
Then $(\triangleright :=0, \, f,
\, \cdot_V)$ is a crossed system if and only if $(V, \cdot)$ is a $4$-algebra, $f: A \times A \to V$ is symmetric and
$$
f(a^2, a^2) + f(a, a)^2 = 0, \qquad x^2 \cdot_V f(a, a) = 0
$$
for all $a\in A$ and $x\in V$.
\end{examples}

Let now ${\mathcal C} \, (A, V) = (\triangleright, \, f, \, \cdot_V)$ be a crossed system of a $4$-algebra $A$ by $V$. Then, the canonical projection $\pi_A : V \# A \to
A$, $\pi_A (x, a) := a$ is a surjective algebra map and
$ {\rm Ker} (\pi_A) = V \times \{0\} \cong V$ is an ideal in the $4$-algebra $V \# A$. Thus, we obtain that the $4$-algebra $V \# A$ is an
extension of the $4$-algebra $A$ by the $4$-algebra $(V, \cdot_V)$
via
\begin{eqnarray} \eqlabel{extenho1}
\xymatrix{ 0 \ar[r] & V \ar[r]^{i_{V}} & V \# \, A
\ar[r]^{\pi_{A}} & A \ar[r] & 0 }
\end{eqnarray}
where $i_V (x) = (x, \, 0)$, for all $x\in V$. Conversely, we have:

\begin{proposition}\prlabel{hocechiv}
Let $A$ be a $4$-algebra, $E$ a vector space and $\pi : E \to A$ an
epimorphism of vector spaces with $V = {\rm Ker} (\pi)$. Then any
$4$-algebra structure $\cdot_E $ which can be defined on $E$ such that $\pi : (E, \cdot_E) \to A$ is a morphism
of $4$-algebras is isomorphic to a crossed product $V \# A$.
Furthermore, the isomorphism of $4$-algebras $(E, \cdot_E) \cong V
\# A$ can be chosen such that it stabilizes $V$ and co-stabilizes
$A$.

In particular, any $4$-algebra structure on $E$ such that $\pi : E \to
A$ is an algebra map is cohomologous to an extension of the form
\equref{extenho1}.
\end{proposition}

\begin{proof}
Let $\cdot_E $ be a $4$-algebra structure of $E$ such that $\pi:
(E, \cdot_E) \to A$ is an algebra map and let $s : A \to E$ be a $k$-linear section of $\pi$,
i.e. $\pi \circ s = {\rm Id}_{A}$. Then $\varphi : V \times A \to E$, $\varphi (x, a) := x + s(a)$ is an isomorphism of
vector spaces with the inverse $\varphi^{-1} (y) = \bigl(y - s (\pi(y)), \, \pi(y)\bigl)$, for all $y\in E$. Using the section $s$
we define the following bilinear maps:
\begin{eqnarray}
\triangleright &=& \triangleright_{s} \, \, : A \times V \to V,
\,\,\,\,
a \triangleright x := s(a) \cdot_E x \eqlabel{acti}\\
f &=& f_s \,\,\, : A \times A \to V, \,\,\,\,
f (a, b) := s(a) \cdot_E s(b) - s(a \cdot b) \eqlabel{coci} \\
&\cdot_V& \,\,\,\,\,\,\,\,\, : V \times V \to V, \,\,\,\, x \cdot_V y : = x\cdot_E y \eqlabel{multip}
\end{eqnarray}
for all $a$, $b\in A$ and $x$, $y\in V$. These are well-defined maps since $\pi$ is an algebra map and $s$ a section of $\pi$.
Using this crossed data $(\triangleright, \, f, \, \cdot_V)$ connecting $A$ and
$V$ we can prove that the unique $4$-algebra structure $\circ$
that can be defined on the direct product of vector spaces $V
\times A$ such that $\varphi : V \times E \to (E, \cdot_E)$ is an
isomorphism of $4$-algebras is given by:
\begin{equation*}
(x, \, a) \circ (y, \, b) := (x \cdot_V y + a \triangleright y + b \triangleright x  +  f(a, \, b), \, a\cdot b)
\end{equation*}
for all $a$, $b\in A$, $x$, $y \in V$. Indeed, let $\circ $ be
such a $4$-algebra structure on $V\times A$. Then we have:
\begin{eqnarray*}
(x, a) \circ (y, b) &=& \varphi^{-1} \bigl(\varphi(x, a) \cdot_E
\varphi(y, b)\bigl) = \varphi^{-1} \bigl(  ( x + s(a)) \cdot_E (y + s(b)) \bigl) \\
&=& \varphi^{-1} ( x \cdot_E y + s(a) \cdot_E y + x \cdot_E s(b) + s(a) \cdot_E s(b) ) \\
&=& \bigl ( x \cdot_E y + s(a) \cdot_E y + s(b) \cdot_E x + s(a) \cdot_E s(b) - s(a\cdot b), \, a\cdot b  \bigl) \\
&=& \bigl(x \cdot_V y + a \triangleright y + b \triangleright x  +  f(a, \, b), \, a\cdot b)
\end{eqnarray*}
as desired. Thus, $\varphi : V \# A \to (E, \cdot_E)$ is an isomorphism of $4$-algebras which stabilizes
$V$ and co-stabilizes $A$ since the diagram
\begin{eqnarray*}
\xymatrix {& V \ar[r]^{i_V} \ar[d]_{Id} & {V
\# \,\, A }
\ar[r]^{\pi_A} \ar[d]^{\varphi} & A \ar[d]^{Id}\\
& V \ar[r]^{i} & {E}\ar[r]^{\pi} & A}
\end{eqnarray*}
is obviously commutative.
\end{proof}

As a first consequence we obtain that any $4$-algebra is isomorphic to a twisted product as constructed in (2) of
\exref{abelian}.

\begin{corollary} \colabel{imptwisted}
Any $4$-algebra $E = (E, \cdot_E)$ is isomorphic to a twisted product $V \# A$ associated to a $4$-algebra $V$ with 
${\rm dim}_k (V) = {\rm dim}_k (E^2)$ and an abelian algebra $A$.
\end{corollary}

\begin{proof}
Indeed, let $(E, \cdot_E)$ be a $4$-algebra and $V := E'$ its derived algebra. Then $A: = E/E'$ is an abelian algebra and
the canonical projection $\pi : E \to A$ is a surjective $4$-algebra map. Now, we apply \prref{hocechiv}.
\end{proof}

\coref{imptwisted} is a tool that can be used for the classification of $4$-algebras of a given finite dimension.
Let us explain briefly how it works: let $E$ be an $(m+n)$-dimensional $4$-algebra. An invariant of such algebras is the dimension $m$ of the derived algebra. Thus, we
have to fix a positive integer $m$. Based on \coref{imptwisted} we obtain that
$E$ is isomorphic to a twisted product $V \# A$, where $V$ is an $m$-dimensional
$4$-algebra and $A$ is a vector space (viewed with the abelian $4$-algebra structure) of dimension $n$ (explicitely, the multiplication on $V \# A$ is given by \equref{twpr222} in \exref{abelian}).

\begin{example} \exlabel{unumic}
Let $V = (V, \cdot_V)$ be an $m$-dimensional $4$-algebra having $\{e_1, \cdots, e_m\}$ as a basis. Then any $(m+1)$-dimensional
$4$-algebra having the derived algebra $V$ is isomorphic to the $4$-algebra having $\{e_i, \, f_0 \, | \, i = 1, \cdots, m \}$ as a basis and
the multiplication given for any $i$, $j=1, \cdots, m$ by:
\begin{equation} \eqlabel{twpr222bb}
e_i \circ e_j := e_j \cdot_V e_j, \quad e_i \circ f_0 := \xi (e_i), \quad
f_0 \circ f_0 := F
\end{equation}
for some pair $(F, \xi) \in V \times {\rm End}_k (V)$ satisfying the following compatibility conditions
$$
F^2 = 0, \qquad x\cdot_V F + 2\, x^2 \cdot_V \xi (x)  + 2\, \xi(x)^2 + 2\, \xi (x) \cdot_V F = 0
$$
for all $x\in V$.
\end{example}

The semidirect product of $4$-algebras characterizes split epimorphisms
in the category on $4$-algebras, as we mentioned in \exref{abelian}:

\begin{corollary} \colabel{splialg}
An algebra map $\pi: B \to A$ between two $4$-algebras $A$ and $B$ is a split epimorphism in the
category of $4$-algebras if and only if there exists an isomorphism
of $4$-algebras $B \cong V\ltimes A$, where $V = {\rm Ker} (\pi)$
and $V \ltimes A$ is a semidirect product of $4$-algebras as constructed in \exref{abelian}.
\end{corollary}

\begin{proof}
For a semidirect product $V\ltimes A$, the canonical projection $p_A : V \ltimes A \to A$, $p_{A}(x,
\, a) = a$ has a section that is an algebra map defined by $s_A
(a) = (0, \, a)$, for all $a\in A$. Conversely, let $s: A \to B$
be an algebra map with $\pi \circ s = {\rm Id}_A$. Then, the
bilinear map $f_s$ given by \equref{coci} from  \prref{hocechiv} is the trivial map and hence the corresponding
crossed product $V \# A$ reduces to a semidirect product $V \ltimes A$.
\end{proof}

Now we shall describe the morphisms that stabilize $V$ between two arbitrary crossed products $V \# A$ and $V \#' A$, for two reasons:
to compute the Galois group ${\rm Gal}(V \# A / A )$ of the extension $V \hookrightarrow V \# A$ and then to answer the classification
part of the GE problem for $4$-algebras.

\begin{lemma}\lelabel{HHH}
Let ${\mathcal C} \,  (A, V) = (\triangleright, \, f, \, \cdot_V)$
and ${\mathcal C}' \, (A, V) = (\triangleright ', \, f', \,
\cdot_V ')$ be two crossed systems of a $4$-algebra $A$ by $V$, with $V \# A$, respectively $V
\# ' A$, the corresponding crossed products. If $\cdot_V  = \cdot_V '$, then there exists a
bijection between the set of all $4$-algebra morphisms $\psi: V \#
A \to V \# ' A$ which stabilize $V$ and the set of all pairs $(r, \alpha) \in {\rm Hom}_k (A, \, V) \times {\rm End}_{\rm alg} (A)$,
where $r: A \to V$ is a linear map, $\alpha: A \to A$ is an algebra morphism satisfying the following compatibilities for all $a$, $b \in A$ and $x \in V$:
\begin{enumerate}
\item[(M1)] $a \triangleright x = \alpha (a) \triangleright ' x + r(a) \cdot_V ' x$;
\item[(M2)] $f(a, \,b)  = f '(\alpha(a), \,
\alpha(b)) + \alpha (a) \triangleright ' r(b) + \alpha(b) \triangleright ' r(a)  + r(a)
\cdot_V ' r(b) - r(a\cdot b) $
\end{enumerate}
Under the above bijection the $4$-algebra morphism $\psi =
\psi_{(r, \alpha)}: V \# A \to V \# ' A$ corresponding to the pair $(r, \alpha)$ is
given by $\psi(x, a) = (x + r(a), \, \alpha(a))$, for all $a \in A$, $x
\in V$. Furthermore, $\psi = \psi_{(r, \alpha)}$ is an isomorphism of $4$-algebras if and only if $\alpha$ is bijective and
$\psi = \psi_{(r, \alpha)}$ co-stabilize $A$ if and only if $\alpha = {\rm Id}_A$.
\end{lemma}

\begin{proof} First of all we observe that for any linear map $\psi: V \,\# \, A \to V \,\# \, ' A$ that makes the following diagram commutative
$$
\xymatrix {& {V} \ar[r]^{i_{V}} \ar[d]_{Id_{V}}
& {V \,\# \, A}\ar[d]^{\psi}\\
& {V} \ar[r]^{i_{V}} & {V \,\# \, ' A}}
$$
there exists a unique pair of linear maps $(r, \alpha) \in {\rm Hom}_k (A, \, V) \times {\rm End}_k (A)$ such that $\psi(x, a) = (x + r(a), \, \alpha(a) )$, for all $a \in A$,
and $x \in V$. Let $\psi = \psi_{(r, \alpha)}$ be such a linear map,
i.e. $\psi(x, a) = (x + r(a), \, \alpha(a) )$, for some linear maps $r: A
\to V$, $\alpha: A \to A$. We will prove that $\psi$ is a morphism of
$4$-algebras if and only if $\cdot_V  = \cdot_V '$, $\alpha: A \to A$ is an algebra map and the compatibility conditions
(M1)-(M2) hold. To this end, it is enough to prove that the
compatibility
\begin{equation}\eqlabel{Liemap}
\psi \bigl((x, a) \circ (y, b) \bigl) = \psi(x, a) \circ' \psi(y,
b)
\end{equation}
holds on all generators of $V \,\# \, A$. We leave out the detailed computations and only indicate the main steps of this verification.
First, it is easy to see that \equref{Liemap} holds for the pair $(x, 0)$, $(y, 0)$ if and only if $\cdot_V  = \cdot_V '$. Secondly, we
can prove that \equref{Liemap} holds for the pair $(a, 0)$, $(0, b)$ if and only if $\alpha: A \to A$ is an algebra maps and (M2) hold. Finally, \equref{Liemap}
holds for the pair $(x, 0)$, $(0, a)$ if and only if (M1) holds. The last two statements are elementary: we just note that if $\alpha: A
\to A$ is bijective, then $\psi_{(r, v)}$ is an isomorphism of $4$-algebras with the inverse given for any $a \in A$ and $x
\in V$ by:
$$
\psi_{(r, \alpha)}^{-1}(x, \, a) = \bigl(x - r(\alpha^{-1}(a)), \,
\alpha^{-1}(a)\bigl). 
$$
This finishes the proof.
\end{proof}

It can be easily observed from the proof of \leref{HHH} that if $\cdot_V  \neq \cdot_V '$, then there are no $4$-algebra morphisms $\psi: V \#A \to V \# ' A$ which stabilize $V$. As a first application of \leref{HHH} we can describe the Galois group ${\rm Gal} \, (V \# A/ V )$ of the extension $V \hookrightarrow V \# A$, consisting of all automorphisms of the crossed product $4$-algebra $V \# A$ acting as identity on $V$. Let ${\mathcal C} \,  (A, V) = (\triangleright, \, f, \, \cdot_V)$ be a crossed system of a $4$-algebra $A$ by $V$. Let ${\mathcal G} \, _A ^V (\triangleright, \, f, \, \cdot_V)$ be the set of all pairs $(r, \alpha) \in {\rm Hom}_k (A, \, V) \times {\rm Aut}_{\rm Alg} (A)$ consisting of a linear map $r : A \to V$ and an automorphism $\alpha : A \to A$ of the $4$-algebra $(A, \cdot)$
such that:
\begin{eqnarray*}
a \triangleright x &=& \alpha (a) \triangleright  x + r(a) \cdot_V  x \\
f(a, \,b) &=& f (\alpha(a), \, \alpha(b)) + \alpha (a) \triangleright  r(b) + \alpha(b) \triangleright  r(a)  + r(a)
\cdot_V  r(b) - r(a\cdot b)
\end{eqnarray*}
for all $x\in V$, $a$, $b\in A$. Then, we can easily prove that ${\mathcal G} \, _A ^V (\triangleright, \, f, \, \cdot_V)$  has a group structure via the multiplication given by:
\begin{equation}\eqlabel{gal}
(r, \, \alpha) \bullet (r', \,  \alpha') := (r'+ r\circ \alpha' , \, \alpha\circ \alpha' )
\end{equation}
for all $(r, \, \alpha)$, $(r', \,  \alpha') \in {\mathcal G} \, _A ^V (\triangleright, \, f, \, \cdot_V)$ and, moreover, ${\mathcal G} \, _A ^V (\triangleright, \, f, \, \cdot_V)$ is a subgroup in the semidirect product of groups
${\rm Hom}_k (A, \, V) \ltimes {\rm GL}_k (A)$ as defined by \equref{semidirect0}.  Applying \leref{HHH} for
$(\triangleright ', \, f', \, \cdot_V ') : = (\triangleright , \, f, \, \cdot_V )$ we obtain:

\begin{corollary} \colabel{galgr}
Let ${\mathcal C} \,  (A, V) = (\triangleright, \, f, \, \cdot_V)$ be a crossed system of a $4$-algebra $A$ by $V$.
Then the map defined by:
$$
\vartheta : ({\mathcal G} \, _A ^V (\triangleright, \, f, \, \cdot_V), \, \bullet) \to {\rm Gal} \, (V \# A/ V ), \qquad \vartheta (r, \alpha) (x, \, a) :=
\bigl ( x + r(a), \,  \alpha(a) \bigl)
$$
for all $(r, \alpha) \in {\mathcal G} \, _A ^V (\triangleright, \, f, \, \cdot_V)$ and $(x, a) \in V\times A$ is an isomorphism of groups.
\end{corollary}

Arising from \leref{HHH} we introduce the following concept needed for the classification part of the GE-problem :

\begin{definition}\delabel{echiaa}
Let $A$ be a  $4$-algebra and $V$ a vector space. Two crossed
systems ${\mathcal C} \,  (A, V) = (\triangleright, \, f, \, \cdot_V)$
and ${\mathcal C}' \, (A, V) = (\triangleright ', \, f', \, \cdot_V ')$ of $A$ by $V$ are called \emph{cohomologous}, and we denote this by
${\mathcal C} \, (A, V) \approx {\mathcal C}' \, (A, V)$, if and only if $\cdot_V =
\cdot_V '$ and there exists a linear map $r: A \to V$ such that
for any $a$, $b \in A$, $x \in V$ we have:
\begin{eqnarray}
a \triangleright x &=& a \triangleright ' x + r(a) \cdot_V ' x \eqlabel{compa2}\\
f(a,\, b) &=& f '(a, \, b) + a \triangleright '
r(b) + b \triangleright ' r(a) + r(a) \cdot_V ' r(b) - r(a\cdot b). 
\eqlabel{compa3}
\end{eqnarray}
\end{definition}

As a conclusion of the results obtained so far we obtain the theoretical answer to the GE-problem for $4$-algebras:

\begin{theorem}\thlabel{main1222}
Let $A$ be a $4$-algebra, $E$ a vector space and $\pi : E \to A$ a
linear epimorphism of vector spaces with $V = {\rm Ker} (\pi)$.
Then $\approx$ defined in \deref{echiaa} is an equivalence
relation on the set ${\mathcal C} {\mathcal S} (A, \, V)$ of all
crossed systems of $A$ by $V$. If we denote by ${\mathbb G}
{\mathbb H}^{2} \, (A, \, V) := {\mathcal C}\mathcal{S} (A, V)/
\approx $, then the map
\begin{equation} \eqlabel{thprinc}
{\mathbb G} {\mathbb H}^{2} \, (A, \, V) \to {\rm Gext} \, (A, \,
E), \qquad \overline{(\triangleright, \, f, \, \cdot_V)}
\, \longmapsto \, V \#_{(\triangleright, \, f, \, \cdot_V)} \, A
\end{equation}
is bijective, where $\overline{(\triangleright, \, f, \,
\cdot_V)}$ denotes the equivalence class of $(\triangleright, \,
f, \, \cdot_V)$ via $\approx$.
\end{theorem}

\begin{proof} Using the last statement of \leref{HHH} we obtain that ${\mathcal C} \, (A, V) \approx {\mathcal C}' \, (A, V)$
if and only if there exists an isomorphism of $4$-algebras $V \# A \cong V \#' A$ that stabilize $V$ and co-stabilize $A$. Thus, $\approx$ is an equivalence
relation on the set ${\mathcal C} {\mathcal S} (A, \, V)$. The rest of the proof follows from \prref{crossprod} and \prref{hocechiv}.
\end{proof}

The classifying object ${\mathbb G} {\mathbb H}^{2} \, (A, \, V)$ constructed in \thref{main1222} will be called the \emph{global non-abelian cohomology} of $A$ by $V$ and its explicit computation for a given $4$-algebra $A$ and a vector space $V$ is a very difficult problem. The first step
in its calculation is inspired by the way the equivalence relation $\approx$ is expressed in
\deref{echiaa}: two different $4$-algebra structures $\cdot_V$ and $\cdot_V^{'}$ on
$V$ give rise to two different equivalence classes with respect to
the relation $\approx$ on ${\mathcal C} {\mathcal S} (A, V)$. Hence we can fix $\cdot_{V}$ a
$4$-algebra structure on $V$ and denote by ${\mathcal C}
{\mathcal S}_{\cdot_{V}} (A, \, V)$ the set of pairs
$\bigl(\triangleright, \, f \bigl)$ such that $\bigl(
\triangleright, \, f, \, \cdot_V \bigl) \in {\mathcal C}
{\mathcal S} (A, V)$. Two such pairs $\bigl(\triangleright, \,
f \bigl)$ and $\bigl(\triangleright', \, f' \bigl)
\in {\mathcal C} {\mathcal S}_{\cdot_{V}} (A, \, V)$ will be
called \emph{$\cdot_V$-cohomologous} and will be denoted by
$\bigl(\triangleright, \, f \bigl) \, \approx_{\cdot_V}
\bigl(\triangleright', \, f' \bigl)$ if there exists a linear map $r: A \to V$ such that
for any $a$, $b \in A$, $x \in V$ we have:
\begin{eqnarray}
a \triangleright x &=& a \triangleright ' x + r(a) \cdot_V  x \eqlabel{compa2aa}\\
f(a,\, b) &=& f '(a, \, b) + a \triangleright '
r(b) + b \triangleright ' r(a) + r(a) \cdot_V  r(b) - r(a\cdot b). 
\eqlabel{compa3ab}
\end{eqnarray}
Then $\approx_{\cdot_V}$ is an equivalence relation on
${\mathcal C} {\mathcal S}_{\cdot_{V}} (A, \, V)$ and we denote by
${\mathbb H}^{2}_{\rm nab} \, \bigl(A, \, (V, \, \cdot_{V} )\bigl)$ the
quotient set ${\mathcal C} {\mathcal S}_{\cdot_{V}} (A, \, V)/
\approx_{\cdot_V}$ and call it the \emph{non-abelian cohomology} of the $4$-algebras
$(A, \cdot)$ and $(V, \cdot_V)$. The object ${\mathbb H}^{2}_{\rm nab} \, \bigl(A, \, (V, \, \cdot_{V} )\bigl)$
classifies all extensions of the given $4$-algebra $A$ by the given
$4$-algebra $(V, \, \cdot_{V})$ and gives the theoretical answer to
the classical extension problem for $4$-algebras. We record all these observations in the
following result:

\begin{corollary} \colabel{desccompcon}
Let $A$ be a $4$-algebra and $V$ a vector space. Then
\begin{equation}\eqlabel{balsoi}
{\mathbb G} {\mathbb H}^{2} \, (A, \, V) = \, \amalg_{\cdot_{V}}
\, {\mathbb H}^{2}_{\rm nab} \, \bigl(A, \, (V, \, \cdot_{V} )\bigl)
\end{equation}
where the coproduct on the right hand side is in the category of
sets over all possible $4$-algebra structures $\cdot_{V}$ on the
vector space $V$. Furthermore, for a given $4$-algebra structure on $V$ the map
\begin{equation} \eqlabel{clasextprob}
{\mathbb H}^{2}_{\rm nab} \, \bigl(A, \, (V, \, \cdot_{V} )\bigl) \to {\rm
Ext} \, (A, \, (V, \, \cdot_V) ), \qquad
\overline{\overline{(\triangleright, \, f)}} \,
\longmapsto \, V \#_{(\triangleright, \, f, \,
\cdot_V)} \, A
\end{equation}
is bijective, where ${\rm Ext} \, (A, \, (V, \, \cdot_V) )$ is the
set of equivalence classes of all $4$-algebras that are extensions
of the $A$ by $(V, \cdot_V)$ and
$\overline{\overline{(\triangleright, \, f)}}$ denotes the
equivalence class of $(\triangleright, \, f)$ via
$\approx_{\cdot_V}$.
\end{corollary}

We continue our investigation of the object ${\mathbb G} {\mathbb H}^{2} \, (A, \, V)$ observing that
among all components of the coproduct in \equref{balsoi}
the simplest one is that corresponding to the trivial $4$-algebra
structure on $V$, i.e. $x \cdot_V y := 0$, for all $x$, $y \in V$
which we denoted by $V_0$. Let
${\mathcal C} {\mathcal S}_0 (A, \, V_0)$ be the set of all pairs
$\bigl(\triangleright, \, f \bigl)$ such that $\bigl(
\triangleright, \, f, \, \cdot_V := 0 \bigl) \in {\mathcal
C} {\mathcal S} (A, V)$. As shown in (3) of \exref{abelian}, a pair
$\bigl(\triangleright, \, f \bigl) \in {\mathcal C} {\mathcal S}_0 (A, \, V_0)$ if and only if
$(V, \triangleright)$ is an $A$-module and $f: A \times A \to V$ is symmetric abelian $2$-cocycles, i.e.
\begin{equation} \eqlabel{celmais}
a^2 \triangleright (a \triangleright x) = 0, \qquad  f(a^2, a^2) + 2\, a^2 \triangleright f(a, a) = 0
\end{equation}
for all $a\in A$ and $x\in V$. Applying now \deref{echiaa} for the trivial multiplication $\cdot_V := 0$ we obtain that
two pairs $\bigl(\triangleright, \, f \bigl)$ and
$\bigl(\triangleright', \, f' \bigl) \in {\mathcal C} {\mathcal S}_0 (A, \, V_0)$ are \emph{$0$-cohomologous}
$\bigl(\triangleright, \, f \bigl) \approx_0
\bigl(\triangleright', \, f' \bigl)$ if and only if
$\triangleright = \triangleright'$ and there exists a linear map
$r: A \to V$ such that for all $a$, $b\in A$ we have:
\begin{equation}\eqlabel{cazslab}
f (a, \, b) = f '(a, \, b) + a \triangleright  r(b) + b \triangleright r(a) - r(a\cdot b). 
\end{equation}
The equality $\triangleright = \triangleright'$ shows that two
different $A$-module structures on $V$ give different
equivalence classes in the classifying object ${\mathbb H}_{\rm nab}^{2} \,
\bigl(A, \, V_0 \bigl)$. Thus, we can apply the same strategy as
before for computing  ${\mathbb H}_{\rm nab}^{2} \, \bigl(A, \, V_0 \bigl)$:
we fix $(V, \, \triangleright)$ a $A$-module structure on $V$ and consider the set ${\rm Z}^2_{\triangleright} \, (A, \, V_0) $
of all symmetric abelian $2$-cocycle, i.e. symmetric bilinear maps $f : A \times A \to V$ such that
\begin{eqnarray*} \eqlabel{cocycleb}
f(a^2, a^2) + 2\, a^2 \triangleright f(a, a) = 0
\end{eqnarray*}
for all $a$, $b\in A$.  Two symmetric abelian $2$-cocycles
$f$ and $f'$ are $0$-cohomologous $f \approx_0 f'$ if and only if there exists a
linear map $r: A \to V$ such that \equref{cazslab} holds. Then
$\approx_0$ is an equivalence relation on ${\rm Z}^2_{\triangleright} \, (A, \, V_0) $ and
the quotient set ${\rm Z}^2_{\triangleright} \, (A, \, V_0)/ \approx_0$, which we will
denote by ${\rm H}^2_{\triangleright} \, (A, \, V_0)$, plays the
role of the second cohomological group from the theory of groups or Lie algebras. All in all, we have obtained the following results which
classifies all extensions of a $4$-algebra $A$ by an abelian
algebra $V_0$.

\begin{corollary}\colabel{cazuabspargere}
Let $A$ be a $4$-algebra and $V$ a vector space viewed with the trivial $4$-algebra structure $V_0$. Then:
\begin{equation}\eqlabel{balsoi2}
{\mathbb H}_{\rm nab}^{2} \, \bigl(A, \, V_0 \bigl) \, = \, \amalg_{\triangleright} \, {\rm H}^2_{\triangleright} \, (A, \,
V_0)
\end{equation}
where the coproduct on the right hand side is in the category of sets over all possible $A$-module structures $\triangleright$ on the vector space $V$.
\end{corollary}

\section{Applications and examples} \selabel{exappl}

The computation of the classifying object ${\mathbb G} {\mathbb H}^{2} \, (A, \, V)$ as constructed in
\thref{main1222} and \coref{desccompcon} is a very challenging problem. In the following we shall
compute ${\mathbb G} {\mathbb H}^{2} \, (A, \, k)$ and ${\mathbb G} {\mathbb H}^{2} \, (k, \, V)$. First we need the following:

\begin{proposition}\prlabel{cohocflag}
Let $A$ be a $4$-algebra. Then there exists a bijection between the set ${\mathcal C}{\mathcal S} \, (A, \, k)$  of all
crossed systems of $A$ by $k$ and the set ${\mathcal C} {\mathcal F} \, (A)$ consisting of all pairs $(\lambda, \, f)$, where
$\lambda :A \to k$ is linear map, $f: A\times A \to k$ a symmetric bilinear form on $A$ satisfying the following compatibilities conditions for any $a\in A$:
\begin{equation}\eqlabel{coflag11}
f(a^2, \, a^2) + 2\, \lambda(a^2) f(a, \, a) = 0, \qquad \lambda(a) \lambda(a^2) = 0. 
\end{equation}
The correspondence is given such that the crossed crossed product
$k \# A$ associated to the pair $(\lambda, \, f)\in {\mathcal C} {\mathcal F} \, (A) $ is the $4$-algebra denoted by $A_{(\lambda,
f)} := k \times A$ with the multiplication given for any $a$, $b \in A$ and
$x$, $y\in k$ by:
\begin{equation} \eqlabel{patratnul}
(x, a) \circ (y, b) = \bigl( \lambda (a) y + \lambda (b) x + f (a, b), \, a\cdot b  \bigl). 
\end{equation}
Furthermore, a $4$-algebra $B$ has a surjective algebra map $B \to A\to 0$ whose kernel is $1$-dimensional if and
only if $B$ is isomorphic to $A_{(\lambda, \, f)}$, for some
$(\lambda, f) \in {\mathcal C} {\mathcal F} \, (A)$.
\end{proposition}

\begin{proof}
Since $k$ has dimension $1$, any $4$-algebra structure $\cdot_k$ on $k$ is the abelian one, $x\cdot_k y = 0$, for all
$x$, $y\in k$. Hence, there exists a bijection between the set of all crossed datums
$\bigl( \triangleright, \, f, \, \cdot_k \bigl)$ of $A$ by
$k$ and the set of pairs $(\lambda, \, f)$ consisting of a linear map $\lambda: A \to k$ and a bilinear map $f: A
\times A \to k$. The bijection is given such that $(\triangleright, \, f)$ corresponding to
$(\lambda, \, f)$ is defined by $a\triangleright x :=
\lambda (a) x$ for all $a \in A$ and $x \in k$. Now the axiom (CS1) of \prref{crossprod}
is equivalennt to $f$ being symmetric, while the axioms (CS2) and (CS3) are equivalent to \equref{coflag11}. The
algebra $A_{(\lambda, \, f)}$ is just the crossed product $k \# A$ associated to this context and the last statement follows from
\prref{hocechiv}.
\end{proof}

Let $(\lambda, f) \in {\mathcal C} {\mathcal F} \, (A)$ be a pair as in \prref{cohocflag}.
We shall explicitly describe the multiplication of the $4$-algebra $A_{(\lambda, \, f)}$.
We will see the elements of $A$ as elements in $k \times A$ through the identification $a = (0, \,
a)$ and denote by $g := (1, \, 0_A) \in k\times A$. Let $\{ e_i \, | \, i \in I \}$ be a basis of $A$ as a vector space
over $k$. Then, the $4$-algebra $A_{(\lambda, f)}$ is the vector space having  $\{ g,
\, e_i \, | \, i \in I \}$ as a basis and the multiplication $\circ $ is given for any $i$, $j\in I$
by:
\begin{equation}\eqlabel{primaalg}
g^2 = 0, \quad  e_i \circ g = g \circ e_i =
\lambda(e_i) \, g, \quad e_i \circ e_j = e_i \cdot e_j + f(e_i, \, e_j) \, g.
\end{equation}

In order to compute ${\mathbb G} {\mathbb H}^{2} \, (A, \, k)$,
we observe first that the equivalence relation given by \equref{compa2} and \equref{compa3} from \deref{echiaa}, written for the set
${\mathcal C} {\mathcal F} \, (A)$, via the bijection ${\mathcal C}{\mathcal S} \, (A, \, k) \cong {\mathcal C} {\mathcal F} \, (A)$
proven in \prref{cohocflag} takes the following form: $(\lambda, f) \approx (\lambda', f')$ if and only if $\lambda = \lambda'$ and there exist
a linear map $r: A\to k$ such that
\begin{equation}\eqlabel{flagrel}
f(a, \, b) = f '(a, \, b) + \lambda ' (a)  r(b) + \lambda '(b) r(a) - r(a\cdot b)
\end{equation}
for all $a$, $b\in A$. Thus, we obtain that ${\mathbb G} {\mathbb H}^{2} \, (A, \, k) \cong {\mathcal C} {\mathcal F} \, (A)/ \approx$.
We continue our investigation since the equality $\lambda = \lambda'$ in the above equivalence relation shows that two different $\lambda$ and $\lambda'$
give different equivalence classes in the classifying object ${\mathcal C} {\mathcal F} \, (A)/ \approx$. Hence, we can fix a linear map
$\lambda: A\to k$ such that $\lambda(a) \lambda(a^2) = 0$, for all $a\in A$. We denote by  ${\rm Z}^2_{\lambda} \, (A, \, k)$ the set of all
\emph{$\lambda$-cocycles}: that is, the set of all symmetric bilinear
maps $f : A \times A \to k$ satisfying the first compatibility condition of \equref{coflag11}. Two $\lambda$-cocycles $f$,
$f' : A \times A \to k$ are equivalent $f
\approx^{\lambda} f'$ if and only if there exists a linear
map $r: A \to k$ such that
\begin{equation*}\eqlabel{flagrelb}
f(a, \, b) = f '(a, \, b) + \lambda  (a)  r(b) + \lambda (b) r(a) - r(a\cdot b)
\end{equation*}
for all $a$, $b\in A$. We denote ${\rm H}^2_{\lambda} \, (A, \, k) := {\rm Z}^2_{\lambda} \, (A, \,
k)/ \approx^{\lambda}$ and we record all the above results in the following decomposition of ${\mathbb G} {\mathbb H}^{2} \, (A, \, k)$, which is a special case of
\coref{cazuabspargere} applied for $V : = k$:

\begin{corollary} \colabel{hoccal1ab}
Let $A$ be a $4$-algebra. Then,
\begin{equation} \eqlabel{astaedefol}
{\mathbb G} {\mathbb H}^{2} \, (A, \, k) \, \cong \, {\mathcal C} {\mathcal F} \, (A)/ \approx \,
\, \cong \, \amalg_{\lambda} \,\, {\rm H}^2_{\lambda} \, (A, \, k)
\end{equation}
where the coproduct on the right hand side is in the category of
sets over all possible linear maps $\lambda: A\to k$ satisfying
$\lambda(a) \lambda(a^2) = 0$, for all $a\in A$.
\end{corollary}

\begin{example} \exlabel{calcexp}
Let $n$ be a positive integer and $A$ the $4$-algebra having $\{e_1, \cdots, e_{n+1} \}$ as a basis and the
multiplication given by $e_1 \cdot e_2 := e_{n+1}$ and the other products of basis elements are zero. Then we can prove that
$$
{\mathcal C} {\mathcal F} \, (A) \cong \, k^n \times {\rm Sym}_{n+1}^{0} \, (k)
$$
where we denoted by ${\rm Sym}_{n+1}^{0} \, (k)$ the vector space of all $(n+1)\times (n+1)$-symmetric matrices
$(f_{ij})$ such that $f_{n+1,n+1} := 0$. The bijection is given such that $(\lambda, \, f) \in {\mathcal C} {\mathcal F} \, (A)$ associated to a pair
$\bigl( \lambda_i \bigl) \times \, (f_{ij}) \in k^n \times {\rm Sym}_{n+1} ^{0} \, (k)$
is given by:
$$
\lambda (e_t) := \lambda_t, \quad \lambda (e_{n+1}) := 0, \qquad f(e_i, \, e_j) := f_{ij}
$$
for all $t = 1, \cdots, n$ and $i$, $j = 1, \cdots, n+1$. In particular, we obtain that
\begin{equation*} \eqlabel{astaedefolbb}
{\mathbb G} {\mathbb H}^{2} \, (A, \, k) \, \cong \, \amalg_{\lambda} \,\, {\rm Sym}_{n+1} ^{0} \, (k) / \approx_{\lambda}
\end{equation*}
where the coproduct on the right hand side is taken over all $\lambda = (\lambda_1, \cdots, \lambda_n) \in k^n$ and
${\rm Sym}_{n+1} ^{0} \, (k) / \approx_{\lambda}$ is the quotient set of ${\rm Sym}_{n+1} ^{0} \, (k)$ via the following equivalence relation:
$(f_{ij}) \approx_{\lambda} (f_{ij}')$ if and only if there exists a linear map $r: A \to k$ such that
$$
f_{ij} = f_{ij}' + \lambda_i r (e_j) + \lambda_j r (e_i) - r(e_i \cdot e_j)
$$
for all $i$, $j = 1, \cdots, n+1$ (with $\lambda_{n+1} := 0$).
\end{example}

Now we describe the opposite case, namely ${\mathbb G} {\mathbb H}^{2} \, (k, \, V)$.

\begin{proposition}\prlabel{cohocflagbb}
Let $V$ be a vector space. Then there exists a bijection between the set ${\mathcal C}{\mathcal S} \, (k, \, V)$ of all
crossed systems of $k$ by $V$ and the set ${\mathcal C} {\mathcal T} \, (V)$ consisting of all triples $(\theta, \, F, \, \cdot_V)$, where
$\theta : V \to V$ is linear map, $F \in V $ and $\cdot_V : V \times V$ is a $4$-algebra structure on $V$ satisfying the following compatibility conditions for any $x\in V$:
\begin{equation}\eqlabel{trif}
F^2 = 0, \quad \theta (x) \cdot_V \, x^2  = \theta (x) \cdot_V \, F = 0, \quad 2\, \theta (x)^2 + x^2 \cdot_V \, F = 0. 
\end{equation}
The correspondence is given such that the crossed product
$V \# k$ associated to the triple $(\theta, \, F, \, \cdot_V) \in {\mathcal C} {\mathcal T} \, (V) $ is the $4$-algebra denoted by $V_{(\theta, \, F, \, \cdot_V)} := V \times k$ with the multiplication given for any $a$, $b \in k$ and $x$, $y\in V$ by:
\begin{equation} \eqlabel{patratnul}
(x, a) \circ (y, b) = \bigl(x \cdot_V y + a \theta (y) + b \theta(x) + ab F , \, 0 \bigl). 
\end{equation}
\end{proposition}

\begin{proof}
We leave it to the reader since it is similar to the one of \prref{cohocflag}, taking into account that the only $4$-algebra structure on $A := k$ is the abelian one.
\end{proof}

We fix now $\cdot_V$ a $4$-algebra structure on the vector space $V$ and denote by ${\mathcal T}_{\cdot_V} \, (V)$ the set of all pairs
$(\theta, \, F) \in {\rm End}_k (V) \times V$ satisfying the compatibility conditions \equref{trif}. Two pairs $(\theta, \, F)$ and $(\theta', \, F') \in {\mathcal T}_{\cdot_V} \, (V)$ are $\cdot_V$-cohomologous $(\theta, \, F) \approx_{\cdot_V} (\theta', \, F')$ if and only if there exists an element $r \in V$ such that
$$
\theta (x) = \theta' (x) + r \cdot_V x, \qquad F = F' + 2 \, \theta' (r) + r^2
$$
for all $x\in V$. Applying \coref{desccompcon} we obtain:

\begin{corollary} \colabel{desccompcon2b}
Let $V$ be a vector space. Then
\begin{equation}\eqlabel{balsoib}
{\mathbb G} {\mathbb H}^{2} \, (k, \, V) = \, \amalg_{\cdot_{V}} \, {\mathcal T}_{\cdot_V} \, (V) / \approx_{\cdot_V}
\end{equation}
where the coproduct on the right hand side is in the category of sets over all possible $4$-algebra structures $\cdot_{V}$ on $V$.
\end{corollary}

In the last part we shall apply our results to metabelian $4$-algebras.
We recall that a $4$-algebra $E$ is called metabelian if the
derived algebra $E'$ is an abelian subalgebra of $E$, i.e.
$(a\cdot b) \cdot (c \cdot d) = 0$, for all $a$, $b$, $c$, $d \in
E$. Let $I$ be an ideal of a $4$-algebra $E$: then the quotient algebra $E/I$ is
an abelian algebra if and only if $E' \subseteq I$. Thus, $E$ is a
metabelian $4$-algebra if and only if it fits into an exact
sequence of $4$-algebras
\begin{eqnarray} \eqlabel{sirexmetab}
\xymatrix{ 0 \ar[r] & B \ar[r]^{i} & {E} \ar[r]^{\pi} & A \ar[r] &
0 }
\end{eqnarray}
where $B$ and $A$ are abelian $4$-algebras. Indeed, if $E$ is
metabelian we can take $B := E'$ and $A := E/E'$ with the obvious
morphisms. Conversely, assume that a $4$-algebra $E$ is an extension of an
abelian algebra $A$ by an abelian algebra $B$. Since $E/i(B) =
E/{\rm Ker}(\pi) \cong A$ is an abelian algebra, we obtain that
$E' \subseteq i(B) \cong B$. Hence, $E'$ is abelian as a
subalgebra in an abelian algebra, i.e. $E$ is metabelian. Using this observation,
\prref{crossprod} and \prref{hocechiv} we obtain the structure of metabelian $4$-algebras:

\begin{corollary}\colabel{metstucture}
A $4$-algebra $E$ is metabelian if and only if there exists an isomorphism of $4$-algebras $E \cong V \#_{(\triangleright, \,
f)} \, A$, where $A$ and $V$ are two vector spaces and $\triangleright : A \times V \to V$, $f: A\times A \to V$ are two bilinear maps such that $f$ is symmetric. The crossed
product $V \#_{(\triangleright, \, f)} \, A$ is the vector space $V\times A$ with the multiplication given for any $a$, $b\in A$ and $x$, $y\in V$ by:
\begin{equation}\eqlabel{inmetab}
(x, \, a) \circ (y, \, b) := (a \triangleright y + b \triangleright x  +  f(a, \, b), \, 0). 
\end{equation}
\end{corollary}

Let now $A$ and $V$ be two fixed vector spaces viewed with the abelian $4$-algebra structure $A_0$ and $V_0$: in the next step we shall classify all
metabelian $4$-algebras that are extensions of $A$ by $V$, that is we classify all crossed products $V \#_{(\triangleright, \,
f)} \, A$ up to an isomorphism that stabilizes $V$ and co-stabilizes $A$. For this purpose, we have to compute the classifying object
${\mathbb H}^{2}_{\rm nab} \, \bigl(A, \, (V, \, \cdot_{V} )\bigl)$ from \coref{desccompcon} and \coref{cazuabspargere} in the case that
both $\cdot_V$ and $\cdot_A$ are trivial multiplication. Let $\triangleright : A\times V \to V$ be a fixed bilinear map and ${\rm Sym} (A\times A; \, V)$
the set of all symmetric bilinear maps $f: A\times A \to V$. Two elements $f$ and $f' \in {\rm Sym} \, (A\times A; \,  V)$ are called \emph{$\triangleright$-cohomologous} $f \approx_{\triangleright} \, f'$ if and only if there exists a linear map $r: A\to V$ such that:
\begin{equation} \eqlabel{100}
f (a, \, b) = f '(a, \, b) + a \triangleright  r(b) + b \triangleright r(a)
\end{equation}
for all $a$, $b\in A$ (i.e. \equref{cazslab} holds for the trivial multiplication on $A$). We denote by
${\rm H}^2_{\triangleright} (A_0, \, V_0) := {\rm Sym} \, (A\times A; \, V)/ \approx_{\triangleright}$.
We have obtained the following:

\begin{corollary}\colabel{cazulmeatabclas}
Let $A$ and $V$ be two vector spaces viewed with the abelian
algebra structure $A_0$ and $V_0$. Then there exists a bijection:
\begin{equation} \eqlabel{metcalsex}
{\mathbb H}^{2}_{\rm nab} \, \bigl(A_0, \, V_0 \bigl) \, \cong \,
\amalg_{\triangleright} \, {\rm H}^2_{\triangleright} \, (A_0, \, V_0)
\end{equation}
where the coproduct on the right hand side is taken over all bilinear maps $\triangleright : A \times V
\to V$.
\end{corollary}

\begin{examples} \exlabel{calcexpb}
1. Let $V$ be a vector space with a basis $\{e_i \, | \, i\in I\}$ viewed with the abelian $4$-algebra structure $V_0$. Then
$$
{\mathbb H}^{2}_{\rm nab} \, \bigl(k_0, \, V_0 \bigl) \, \cong \,
\amalg_{g\in {\rm End}_k (V) } \,\,  V/ {\rm Im} (g). 
$$
In particular, any $4$-algebra $E$ containing $V$ as an abelian ideal of codimension $1$ is isomorphic to the $4$-algebra
with the basis $\{e, \, e_i \, | \, i\in I\}$ and the multiplication given for any $i\in I$ by:
$$
e \circ e_i := g(e_i), \qquad e^2 :=f_0
$$
for some $g\in {\rm End}_k (V)$ and $f_0 \in V$. Indeed, since $A := k$ any bilinear map $\triangleright : k \times V \to V$
has the form $\alpha \triangleright x = \alpha g(x)$, for a linear map $g\in {\rm End}_k (V)$ and the correspondence is bijective. Moreover, the set of all symmetric bilinear maps $f: k \times k \to V$ is in bijection with the set of all elements of $V$ (the bijection maps $f$ to $f_0 := f(1, 1))$. The conclusion follows from \coref{cazulmeatabclas} once we observe that the equivalent relation \equref{100} written for the set of all elements $f_0 \in V$ comes down to $f_0 \approx_{\triangleright} \, f_0'$ if and only if
$f_0 - f_0' \in {\rm Im} (g)$. For the last part we use \equref{inmetab} of \coref{metstucture} since any such $4$-algebra is metabelian.

2. The other way around, let $A$ be a vector space with a basis $\{f_j \, | \, j\in J\}$ viewed with the abelian $4$-algebra structure $A_0$. Then,
$$
{\mathbb H}^{2}_{\rm nab} \, \bigl(A_0, \, k_0 \bigl) \, \cong \,
\amalg_{\lambda \in A^*} \,\, {\rm Sym} (A\times A; \, k)/ \approx_{\lambda}
$$
where, for any linear map $\lambda \in A^* ={\rm Hom}_k (A, \, k)$, $\approx_{\lambda}$ is the following equivalent relation of the set of all symmetric bilinear forms on $A$:
$f \approx_{\lambda} \, f'$ if and only if there exists a linear map $r \in A^*$ such that
for any $a$, $b\in A$ :
$$
f (a, \, b) = f '(a, \, b) + r (a) \lambda (b) + r(b) \lambda (a). 
$$
Furthermore, any $4$-algebra $E$ having the derived subalgebra of dimension $1$ is isomorphic to the $4$-algebra
with the basis $\{f, \, f_j \, | \, j\in J\}$ and the multiplication given for any $j$, $l\in J$ by:
$$
f \circ f_j := \lambda (f_j) \, f, \qquad f_j \circ f_l := f(f_j, \, f_l) \, f
$$
for some $\lambda \in A^*$ and $f\in {\rm Sym} (A\times A; \, k)$, where $A$ is the abelian $4$-algebra $A : = E/E'$.
\end{examples}

\textbf{Acknowledgment:} The author warmly thanks the referee for his/her suggestions which improved the first version of the paper.

\section{Competing interests declaration}
The author declares none.

\end{document}